\documentclass[12pt]{amsart}
\usepackage{amsmath}

\def\ol{\overline}
\def\e{\epsilon}
\def\lf{\left}
\def\ri{\right}
\def\a{{\alpha}}
\def\b{{\beta}}

\def\pp{\partial}
\newcommand\R{{\mathbb R}}

\def\tr{\tilde{R}_{ij}}

\def\H{\mathbb{H}}
\def\be{\begin{equation}}
\def\ee{\end{equation}}
\def\ol{\overline}
\def\lf{\left}
\def\ri{\right}
\def\a{{\alpha}}
\def\b{{\beta}}

\def\e{\epsilon}

\def\tr{\text{\rm tr}}

\def\tr{\text{\rm tr}}

\def\be{\begin{equation}}
\def\ee{\end{equation}}
\def\S{\mathbb{S}}

\newtheorem{thm}{Theorem}[section]

\newtheorem{lem}{Lemma}[section]
\newtheorem{prop}{Proposition}[section]
\newtheorem{cor}{Corollary}[section]
\theoremstyle{definition}
\newtheorem{dfn}{Definition}[section]
\theoremstyle{remark}
\newtheorem{rem}{Remark}
\numberwithin{equation}{section}

\newcommand{\ra}{\rightarrow}

\newcommand{\p}[2]{\frac{\partial #1}{\partial #2}}

\begin{document}

\title{limit of quasilocal mass integrals in asymptotically hyperbolic manifolds}
\author{Kwok-Kun Kwong and Luen-Fai Tam*}

\address{The Institute of Mathematical Sciences and Department of
 Mathematics, The Chinese University of Hong Kong,
Shatin, Hong Kong, P. R. China.} \email{kkkwong@math.cuhk.edu.hk}

\address{The Institute of Mathematical Sciences and Department of
 Mathematics, The Chinese University of Hong Kong,
Shatin, Hong Kong, P. R. China.} \email{lftam@math.cuhk.edu.hk}

\thanks{*~Research partially supported by Hong Kong RGC General Research Fund  \#CUHK 403108}

\renewcommand{\subjclassname}
{\textup{2000} Mathematics Subject Classification}
\date{October 2010}

\begin{abstract}
In this paper, we will show that the limit of some quasilocal mass integrals  of the coordinate spheres in an asymptotically hyperbolic (AH) manifold is the mass integral of the AH manifold. This is the analogue of the well known result that the limit of the Brown-York mass of coordinate spheres is the ADM mass in an asymptotically flat manifold.
\end{abstract}
\maketitle
\markboth{}{}

\section{Introduction}

It is known that in an asymptotically flat manifold, the Brown-York quasilocal mass of the coordinate spheres will converge to the ADM mass of the manifold  \cite{fst}, see also \cite{ShiWangWu09,FanKwong}. In this work, we will investigate if there is a corresponding result for asymptotically hyperbolic (AH) manifolds. First we give the meanings of mass of an AH manifold and quasilocal mass. In this work, all manifolds are assumed to be connected and orientable.

We will follow X. D. Wang \cite{Wang} to define asymptotically hyperbolic manifolds as follows:
\begin{dfn}\label{def: g}
 A complete noncompact Riemannian manifold $(M^n,g)$ is said to be asymptotically hyperbolic (AH) if $M$ is the interior of a compact manifold $\ol M$ with boundary $\pp \ol M$ such that:
  \begin{enumerate}
     \item [(i)] there is a smooth function $r$ on $\ol M$ with   $r>0$ on $M$ and $r=0$ on $\pp\ol M$ such that $\ol g=r^2g$ extends as a smooth Riemannian metric on $\ol M$;
     \item [(ii)] $|dr|_{\ol g}=1$ on $\pp \ol M$;
     \item [(iii)] $\pp\ol M$ is  the standard unit sphere $\S^{n-1}$;
         \item[(iv)]  on a collar neighborhood of $\partial \ol M$,
$$g=\sinh^{-2}(r)(dr^2+g_r),$$ with $g_r$ being an $r$-dependent family of metrics on $\mathbb{S}^{n-1}$  satisfying
$$g_r=g_0+\frac{r^n}{n}h+e,$$
  where   $g_0$ is the standard metric, $h$ is a smooth symmetric 2-tensor on $\mathbb{S}^{n-1}$ and $e$ is of order $O(r^{n+1})$, and the asymptotic expansion can be differentiated twice.
   \end{enumerate}
\end{dfn}
Note that the definition is not as general as that in \cite{ChruscielHerzlich}, see also \cite{Zhang2004}. In \cite{Wang}, the following positive mass theorem was proved by Wang (see also \cite{AD,ChruscielHerzlich,Zhang2004})
\begin{thm}\cite[Theorem 2.5]{Wang} \label{thm: wang}
If $(M^n, g)$ is spin, asymptotically hyperbolic and the scalar curvature $R\geq -n(n-1)$, then
$$\int_{\mathbb{S}^{n-1}}\tr_{g_0}(h)d\mu_{g_0}\geq \left|\int_{\mathbb{S}^{n-1}}\tr_{g_0}(h)xd\mu_{g_0}\right|.$$
Moreover equality holds if and only if $(M, g)$ is isometric to the hyperbolic
space $\mathbb{H}^n$.
\end{thm}
We only consider the case that $n=3$, the theorem implies that if $M$ is not isometric to the hyperbolic space, then the vector
$$
\Upsilon=\left(\int_{\mathbb{S}^{n-1}}\tr_{g_0}(h)d\mu_{g_0}, \int_{\mathbb{S}^{n-1}}\tr_{g_0}(h)xd\mu_{g_0}\right)
$$
is a future directed timelike vector in $\R^{3,1}$, the Minkowski space. We may consider $\Upsilon$ as the mass integral for the AH manifold.

We introduce the following quasilocal mass integral for a compact manifold with boundary, similar to the Brown-York mass. Let $(\Omega,g)$ be a three dimensional compact manifold with smooth boundary $\Sigma$. Assume $\Sigma$ is homeomorphic to the standard sphere $\S^2$ such that the mean curvature of $\Sigma$ is positive and the Gaussian curvature of $\Sigma$ is larger than $-1$. Then $\Sigma$ can be isometrically embedded into the hyperbolic space $\H^3$ by a result of Pogorelov \cite{P} and the embedding is unique up to an isometry of $\H^3$. Consider $\H^3$ as the hyperboloid in $\R^{3,1}$
$$\mathbb{H}^3=\{(x^0, x^1, x^2, x^3)\in \R^{3,1}: (x^0)^2-\sum_{i=1}^3 (x^i)^2=1, x^0>0\}.$$
Then the quasilocal mass integral of $\Omega$ is defined as:
$$
\int_{\Sigma}(H_0-H)X
$$
where $H_0$ is the mean curvature of $\Sigma$ in $\mathbb{H}^3$ and $X$ is the position vector in $\R^{3,1}$.

The motivation of this definition is as follows. In \cite{WangYau2007}, M. T. Wang and Yau proved that if the scalar   curvature of $\Omega$ satisfies $R\ge -6$, then there is a future time like vector $W$ such that
$$
\int_{\Sigma}(H_0-H)W
$$
is a future directed non-spacelike vector. $W$ is obtained by solving a backward parabolic equation with a prescribed data at infinity and is not very explicit. Later in \cite{ShiTam2007}, Shi and the second author prove that if $B_o(R_1)$ and $B_o(R_2)$ are two geodesic balls in $\H^3$ such that $B_o(R_1)$ is contained in the interior of $\Sigma$ in $\H^3$ and $\Sigma$ is contained in $B_o(R_2)$,  where $o=(1,0,0,0)\in \H^3\subset \R^{3,1}$, then the result of Wang-Yau is true for  $W(x^0,x^1,x^2,x^3)=(\a x^0,x^1,x^2,x^3)$ with
$$
\a=\coth R_1+\frac1{\sinh R_1}\lf(\frac{\sinh^2R_2}{\sinh^2R_1}-1\ri)^\frac12.
$$
Hence $W$ is close to the position vector. It is an open question whether $W$ can be chosen to be the position vector.

In this work, we consider AH manifolds with the following condition  (with the notations as in Definition \ref{def: g}):

 {\bf Assumption A}:   $\nabla_{\mathbb{S}^{n-1}}e, \nabla^2_{\mathbb{S}^{n-1}}e, \nabla^3_{\mathbb{S}^{n-1}}e, \nabla^4_{\mathbb{S}^{n-1}}e$ with respect to $g_0$ and $\p{e}{r}$ are of order $O(r^n)$.

Let $S_a=\{r=a\}\subset(M, g)$ and let $H$ to be its mean curvature. We identify $S_r$ as the standard sphere $\S^2$ with metric $\gamma_r$ induced from $g$. Then for $r$ small, the Gaussian curvature of $(S_r,\gamma_r)$ is positive where $\gamma_r$ is the induced metric of $g$.

Our main result is the following:
\begin{thm}\label{thm: main}
Let (M, g) be a three-dimensional asymptotically hyperbolic manifold satisfying Assumption A. For all $r$  sufficiently small,  there exists an isometric embedding $X^{(r)}: S_r\ra \mathbb{H}^3\subset\R^{3,1}$ such that
$$
\lim_{r\to0}\int_{S_r}(H_0-H) X^{(r)} d\mu_{\gamma_r}=\frac{1}{2}\left(\int_{\mathbb{S}^{2}}\tr_{g_0}(h)d\mu_{g_0}, \int_{\mathbb{S}^{2}}\tr_{g_0}(h)xd\mu_{g_0}\right)
$$
where  $H_0$ is the mean curvature of $X^{(r)}(S_r)$ in $\H^3$.
\end{thm}

\begin{rem}\label{rem}
From the proof of Theorem \ref{thm: main}, $X^{(r)}$ in the theorem can be chosen by applying an isometry of $\mathbb{H}^3$ fixing $o$ (i.e. $O(3)$) on $\tilde X^{(r)}$, where $\tilde X^{(r)}$ is an embedding of $S_r$ (for small $r$) such that $o$ is the center of a largest geodesic sphere contained in the interior of $\tilde X^{(r)}(S_r)$ (or a smallest geodesic sphere containing $\tilde X^{(r)}(S_r)$ in its interior).
\end{rem}

By applying Theorem \ref{thm: wang} to our result, we have
\begin{cor}\label{cor: main}
Let $(M, g)$ be a three-dimensional asymptotically hyperbolic manifold satisfying Assumption A with the scalar curvature $R\geq -6$, if $Y^{(r)}: S_r\ra \mathbb{H}^3\subset\R^{3,1}$ is an isometric embedding such that
$o$ is the center of a largest geodesic sphere contained in the interior of $Y^{(r)}(S_r)$ (or a smallest geodesic sphere containing $Y^{(r)}(S_r)$ in its interior), then for sufficiently small $r$, the vector
$$\int_{S_r}(H_0-H) Y^{(r)} d\mu_{\gamma_r}$$
is either zero or is future-directed timelike. If $(M, g)$ is not isometric to $\mathbb{H}^3$, then this vector is always non-zero for sufficiently small $r$.
\end{cor}

This paper is organized as follows. In Section \ref{sec: curv}, we will establish some estimates for the various curvatures of $S_r$ and its embedding in the hyperbolic space. In Section \ref{sec: geom}, we will describe some basic results in hyperbolic geometry concerning the radii of the smallest geodesic sphere enclosing a given convex surface and of   the largest geodesic sphere enclosed by it. In Section \ref{sec: emb}, we will normalize the isometric embedding of $S_r$ into the hyperbolic space so that the image of the isometric embedding of $S_r$ is close to a geodesic sphere in the hyperbolic space. We then prove the main results in Section \ref{sec: main results}.

{\sc Acknowledgments}:
The first author would like to thank Chit-Yu Ng and the second author would like to thank Ralph Howard, Yuguang Shi and Andrejs Treibergs for useful and stimulating discussions.

\section{Curvature estimates}\label{sec: curv}

In this section, we always assume $(M^3,g)$ is a three dimensional AH manifold as  in Definition \ref{def: g} such that {Assumption A} is satisfied. Using the notations in Definition \ref{def: g}, let $S_a=\{r=a\}\subset M$. We want to obtain some curvature estimates for $S_r$ which will be used in the proof of the main result. First we will estimate the intrinsic scalar curvature $R$ which is twice the Gaussian curvature  of $S_r$ with the metric $\gamma_r$ induced by $g$.

\begin{lem}\label{lem: R}
The scalar curvature $R$ of $S_r$ with respect to the induced metric from $g$ is given by
$$R=2\sinh^2 r+O(r^5).$$
\end{lem}

\begin{proof}
Recall that $g_r=g_0+\frac{r^3}{3}h+e$. Then $\gamma_r=\sinh^{-2}(r) g_r$ is the induced metric on $S_r$ from $g$. Let $R$ and $\tilde R $ be the scalar curvature of $S_r$ with respect to the metric $\gamma_r$ and $g_r$ respectively. It is easy to see that $R =\sinh^{2}(r)\tilde R$. We claim that
\begin{equation}\label{eq: R}
\tilde R=2+O(r^3).
\end{equation}
The result immediately follows from this claim.

To prove the claim,  let $\{y^i\}_{i=1}^2$ be the local coordinates on the lower hemisphere (say) of $\mathbb{S}^2$ induced by the stereographic projection from the north pole to the plane. Let $\tilde g_{ij}=g_r(\p{}{y^i}, \p{}{y^j}), g_{ij}=g_0(\p{}{y^i}, \p{}{y^j})$ and $\tilde \Gamma_{ij}^k, \Gamma_{ij}^k$ be the Christoffel symbols with respect to $\tilde g_{ij}$ and $g_{ij}$ respectively. Let $\tilde g^{ij}$ and $g^{ij}$ be the inverse of $\tilde g_{ij}$ and $g_{ij}$ respectively.
Then
\begin{equation}\label{eq: R_ijkl}
\tilde R =\sum_{j, k, l}\tilde g^{jk}\tilde R_{ljk}^l \text{ where }
\tilde R_{ijk}^l=\partial _i\tilde \Gamma_{jk}^l
-\partial _j\tilde \Gamma_{ki}^l+\sum_p\tilde \Gamma_{jk}^p \tilde \Gamma_{ip}^l -\sum_p\tilde \Gamma_{ik}^p \tilde \Gamma_{jp}^l,
\end{equation}
and
\begin{equation}\label{eq: R_ijkl0}
2 =\sum_{j, k, l}  g^{jk}  R_{ljk}^l \text{ where }
  R_{ijk}^l=\partial _i  \Gamma_{jk}^l
-\partial _j  \Gamma_{ki}^l+\sum_p  \Gamma_{jk}^p  \Gamma_{ip}^l -\sum_p  \Gamma_{ik}^p   \Gamma_{jp}^l,
\end{equation}
  {  Assumption A}  implies that
$$|\tilde g_{ij}-g_{ij}|=O(r^3), |\tilde g_{ij, k}-g_{ij, k}|=O(r^3)\text{ and }|\tilde g_{ij, kl}-g_{ij, kl}|=O(r^3),$$
where $g_{ij,k}=\p {g_{ij}}{y^k}$ etc.
Hence
\begin{equation}
\tilde\Gamma_{ij}^k-\Gamma_{ij}^k=O(r^3) \text{ and }\partial_i \tilde\Gamma_{jk}^l-\partial_i \Gamma_{jk}^l=O(r^3).
\end{equation}
In view of \eqref{eq: R_ijkl} and \eqref{eq: R_ijkl0}, these imply that $\tilde R_{ijk}^l-R_{ijk}^l=O(r^3)$ and hence $\tilde R -2=O(r^3).$ We conclude that \eqref{eq: R} is true. This completes the proof of the lemma.
\end{proof}

Next, we want to estimate the mean curvature $H$ of $S_r$ with respect to $g$.
\begin{lem}\label{lem: H}
If (M, g) is asymptotically hyperbolic satisfying Assumption A, then the mean curvature of $S_r$ is
$$H=2\cosh r-\frac{1}{2}r^3 \tr_{g_0}h+O(r^4).$$
\end{lem}

\begin{proof}
Let $\{e_j\}_{j=1}^{2}$ be a local orthonormal frame on $(\mathbb{S}^{2}, g_0)$.
The outer unit normal of $S_r$ is $\nu=-\sinh r \p{}{r}$. Denote $g(e_i, e_j)$ by $g_{ij}$ and $g_r(e_i, e_j)$ by $\sigma_{ij}$, then
\begin{equation*}
\begin{split}
H&=\nu\left(\log\sqrt{\det{(g_{ij})}}\right)\\
&=-\sinh r \p{}{r}\left(\log\left(\sinh^{-2}r\sqrt{\det{(\sigma_{ij})}}\right)\right)\\
&=2\cosh r-\frac{1}{2}\frac{\sinh r}{\sqrt{\det{(\sigma_{ij})}}} \p{}{r}\det{(\sigma_{ij})}.
\end{split}
\end{equation*}
It is easy to see that $\det{(\sigma_{ij})}=1+\frac{r^3}{3}\tr_{g_0}h+O(r^4)$ and by the condition $\p{e}{r}=O(r^3)$, that
$\p{}{r}\det{(\sigma_{ij})}=r^2 \tr_{g_0}h+O(r^3)$. Combining these with the above calculation, we can get the result.
\end{proof}

By Lemma \ref{lem: R}, for sufficiently small $r$,  the Gaussian curvature $K$ of $(S_r,\gamma_r)$ is positive. Hence $(S_r,\gamma_r)$  can be isometrically embedded into $\mathbb{H}^3$  which is unique up to an isometry in $\mathbb{H}^3$ by the results of  Pogorelov \cite{P}. Moreover, by the Gauss equation, for an orthonormal frame in $S_r$,
$$ -1+\chi_{11}\chi_{22}-\chi_{12}^2=K>0. $$
Hence the embedded surface which will be denoted by $\Sigma_r$ is strictly convex.
Let $H_0$ be the mean curvature of $\Sigma_r$, we want to estimate $H_0$ and compare it with $H$.

To estimate $H_0$, we will generalize a result on convex compact hypersurfaces in $\R^n$ of Li-Weinstein \cite[Theorem 2]{LW} to compact hypersurfaces in $\H^n$.
\begin{lem}
Suppose $\Sigma$ is a closed convex hypersurface in $\mathbb{H}^n$, $n\geq 3$. If the scalar curvature $R$ of $\Sigma$ satisfies $R+(n-2)(n-3)>0$, then its mean curvature $H_0$ satisfies the inequality
$$ H_0 ^{\;2}\leq \max_{\Sigma}\left(\frac{2\hat R^2-2(n-1)\hat R-\Delta R}{R+(n-2)(n-3)}\right)$$
where $\hat R=R+(n-1)(n-2)$ and $\Delta$ is the Laplacian on $\Sigma$.
\end{lem}
\begin{proof}
We basically follow the ideas from \cite{LW}. Let $\chi$ be the second fundamental form of $\Sigma\subset \mathbb{H}^n$. Let $p\in\Sigma$ be such that $\displaystyle H_0(p)=\max_{\Sigma}H_0$. Let  $\{x^j\}_{j=1}^{n-1}$ be a normal coordinates of $\Sigma$ around $p$ so that $\chi_{ij}=\lambda_i\delta_{ij}$ at $p$. Then at $p$, ${H_0}_{;ij}$ is negative semi definite. Here we use $S_{;k}$ to denote the covariant derivative of $S$ on $\Sigma$ with respect to the induced metric.  Since $\chi_{ij}$ is positive,   at $p$ we have,
\begin{equation}\label{eq: Delta H}
H_0\Delta H_0=\lf(\sum_i \lambda_i\ri)\lf(\sum_i {H_0}_{;ii}\ri)\leq \sum_{i}\lambda_i {H_0}_{;ii}.
\end{equation}
All sums here will have indices from $1$ to $n-1$. Since $\H^n$ has constant curvature, the Codazzi equation implies
 \begin{equation}\label{eq: Codazzi}
 \chi_{ij;k}-\chi_{ik;j}=0.
 \end{equation}
By the Gauss equation, we have
\begin{equation}\label{eq: Gauss}
  R+(n-1)(n-2)=H_0^2-|\chi|^2.
  \end{equation}
  Let $R_{ijkl}$ be the intrinsic curvature tensor of $\Sigma$. At $p$,
\begin{equation*}
\begin{split}
\Delta R
&=2H_0 \Delta H_0+2|\nabla H_0|^2-2|\nabla \chi|^2-2\sum_{i,k}\lambda_i \chi_{ii;kk}\\
&\leq 2\sum_{i,k}\lambda_i\lf( \chi_{kk;ii}-\chi_{ii;kk}\ri)\quad \text{(by \eqref{eq: Delta H} and $ \nabla H_0 =0$)}\\
&= 2\sum_{i,k}\lambda_i\lf( \chi_{kk;ii}-\chi_{ki;ik}\ri)\quad \text{ (by \eqref{eq: Codazzi})}\\
&=2\sum_{i,k,m}\chi_{ij}(R_{kikm}\chi_{mi}+R_{kiim}\chi_{km})\quad \text{(by Ricci identity and \eqref{eq: Codazzi})}\\
&=2\sum_{i,k}R_{kiik}(-\lambda_i^2+\lambda_i\lambda_k)\\
&=2\sum_{i, k}(-1+\lambda_k\lambda_i)(-\lambda_i^2+\lambda_i\lambda_k) \quad{\text{(by the Gauss equation)}}\\
&=2\left((n-1)|\chi|^2-H_0\sum_i\lambda_i^3-H_0^2+|\chi|^4\right).
\end{split}
\end{equation*}
  By \cite[Lemma 2]{LW}, since $\lambda_i>0$,
  $$
 -2 \sum_i\lambda_i^3\le
 \lf(\sum_i\lambda_i\ri)^3-3\lf(\sum_i\lambda_i^2\ri)(\sum_i\lambda_i)
 =H_0^3-3|\chi|^2H_0.
 $$
Plugging this into the above and use \eqref{eq: Gauss}, at $p$,
\begin{equation*}
\begin{split}
\Delta R
&\leq 2(n-1)|\chi|^2+3\hat R H_0^2-2H_0^4+2|\chi|^4-2H_0^2\\
&= 2(n-1)(H_0^2-\hat R)+3\hat R H_0^2-2H_0^4+2(H_0^2-\hat R)^2-2H_0^2\\
&= -(\hat R-2(n-2))H_0^2-2(n-1)\hat R+2\hat R^2.
\end{split}
\end{equation*}
>From this it is easy to see that the lemma is true.
\end{proof}
 Applying the previous lemma to $\Sigma_r$ which is the embedded image of $(S_r,\gamma_r)$, we have:
\begin{cor}\label{cor: H_0 upper bdd}
With the same assumptions and notations as in Lemma \ref{lem: R}, for sufficiently small $r$, the mean curvature $H_0$ of $\Sigma_r$ in $\mathbb{H}^3$ satisfies
$$H_0^{\;2}\leq \max_{S_r}(2R+4-\frac{\Delta R}{R})$$
where $\Delta$ is the Laplacian on $S_r$ under the induced metric, $R=2K$ and $K$ is the Gaussian curvature of $S_r$.
\end{cor}
We now estimate $H_0$.
\begin{lem}\label{lem: H_0}
The mean curvature $H_0$ of $\Sigma_r$ in $\mathbb{H}^3$ is given by
$$H_0=2\cosh r+O(r^5).$$
\end{lem}
\begin{proof}

By the Gauss equation, $2\hat R\leq \hat R+|\chi|^2=H_0^2$ where $\hat R=R+2$ and $\chi$ is the second fundamental form of the embedded $S_r$. So by combining Lemma \ref{lem: R} and Corollary \ref{cor: H_0 upper bdd}, we have
$$4\cosh^2 r+O(r^5)\leq H_0^2\leq 4\cosh^2r+\max_{S_r}\left|\frac{\Delta R}{R}\right|+O(r^5).$$
The proof would be completed if we can show that $\frac{\Delta R}{R}=O(r^5)$. The proof is analogous to that of Lemma \ref{lem: R}. Using the notations in the proof of Lemma \ref{lem: R}, it is easy to see that
\begin{equation}\label{eq: Delta R}
\frac{\Delta R}{R}= \frac{\sinh^4 r }{  R}\Delta_{g_r}\tilde R
\end{equation}
where $\tilde R$ is the scalar curvature with respect to $g_r$. Using Assumption A, we have
$$|\partial^{(k)} _{}\tilde \Gamma_{ij}^l-\partial^{(k)} _{}\Gamma_{ij}^l|=O(r^3) \text{ for }k=0, 1, 2, 3,$$
with respect to the coordinates $\{y^i\}_{i=1}^2$. Together with \eqref{eq: R_ijkl} and \eqref{eq: R_ijkl0}, we conclude that $\partial_i \tilde R-\partial_i  R=O(r^3)\text{ and }\partial^2_{ij} \tilde R-\partial^2_{ij}  R=O(r^3)$. Hence
$$
\Delta_{g_r}\tilde R-\Delta_{g_0}R_0= O(r^3).
$$
As $R_0=2$ is a constant, by \eqref{eq: Delta R} and Lemma \ref{lem: R}, the result follows.
\end{proof}
Combining Lemma \ref{lem: H} and Lemma \ref{lem: H_0}, we have
\begin{cor}\label{cor: H_0-H}
On $S_r$, we have
$$H_0-H=\frac{1}{2}r^3 \tr_{g_0}h+O(r^4).$$
\end{cor}

\section{Inscribed and circumscribed geodesic spheres}\label{sec: geom}

It is well known that  a compact convex hypersurface $\Sigma$ in $\R^n$ can contain and be contained in spheres with radius depending only on the upper and lower bound of principal curvatures $\lambda_i$. In this section, we will describe the corresponding results in $\H^n$, which will be used later. We will sketch the proofs for the sake of completeness whenever we could not locate a reference. We only consider the case $n=3$. The general case is similar. The following is a direct consequence of a result of Ralph Howard \cite[Theorem 4.5]{H}. We would like to thank him for this information.
 \begin{prop}\label{prop: inner}
Let $\Sigma$ be a compact convex surface in $\mathbb{H}^3$ and $\displaystyle \coth b=\max_{x\in \Sigma}\lambda_i(x)\geq \min_{x\in \Sigma}\lambda_i(x)>1$,
then there is a geodesic sphere of radius $b$ which is contained in the interior of $\Sigma$.
\end{prop}
\begin{proof}
By \cite[Theorem 4.5]{H}, since $\lambda_i> 1$ on $\Sigma$, the largest radius of geodesic balls which can roll inside $\Sigma$ is equal to the focal distance of $\Sigma$. It is not hard to see that the focal distance of   $\Sigma$ in $\mathbb{H}^3$ is equal to
$$\min_{x\in \Sigma} \{\rho: \coth \rho=\lambda_i(x), i=1, 2 \}.$$
This can be seen by considering the $\Sigma$-Jacobi field along the inward-pointing geodesics perpendicular to $\Sigma$, see for example \cite{H} p. 474. From this the result follows.
\end{proof}
For circumscribed geodesic spheres of $\Sigma$, we have the following:

\begin{prop}\label{prop: outer}
Let $\Sigma$ be a closed convex surface in $\mathbb{H}^3$ with $\lambda_i>\coth a>1$ on $\Sigma$,
then  there is a geodesic sphere of radius $a$ which contains $\Sigma$ in its interior.
\end{prop}
Since we cannot find an explicit reference for this, we will give more details of the proof. We use the  idea of Andrejs Treibergs \cite{Treibergs}  to give a proof. We would like to thank him for the idea.
To show this, we need the following lemma about convex curves on $\mathbb{H}^2$ which is an extension of Schur's theorem for plane curves.
\begin{lem}\label{Schur-l1} Let $\a$ and $\b$ be two curves in $\H^2$ with same length $l$ parametrized by arc length. Suppose let $\gamma$ be the geodesic from $\a(0)$, $\a(l)$ and $\sigma$ be the geodesic from $\b(0)$ to $\b(l)$. Suppose $\a$ and $\gamma$ bounds a geodesically convex region, and $\b$, $\sigma$ bounds a geodesically convex region. Suppose the geodesic curvature $k$ of $\a$ is larger than the geodesic curvature $\tilde k$ of $\b$ which are assumed to be positive. Then length of $\gamma$ is less than the length of $\sigma$.
\end{lem}
\begin{proof} Let us use the right half plane model for $\H^2$:
$$
\H^2=\{(x,y)\in \R^2|\ x>0\}
$$
with metric $ds^2=\frac{dx^2+dy^2}{x^2}.$
We may assume that $\gamma$ is given by $\gamma(t)=(t,c)$, $a\le t\le b$ and $c$ is a constant. We also assume that $\a$ is below $\gamma$. That is, if $\gamma(s)=(x(s),y(s))$, then $x(s)\le c$. We may assume that $\gamma$ touches the geodesic $(t,c')$ for some $c'$ at $\a(s_0)$ some $0<s_0<l$. Then $\a$ lies between the geodesics $y=c$ and $y=c'$. Move $\b$ such that $\b(s_0)$ touches $y=c'$ at $\b(s_0)$ and such that $\b$ lies above $y=c'$; i.e., $\b$ is in the region $y\ge c'$.

Let $\a(s)=(x(s),y(s))$ and $\b(s)=(\tilde x(s),\tilde y(s))$. Let $\theta(s)$ be the oriented angle from the tangent of $(t,y(s))$ to $\a'(s)$. Define $\tilde \theta(s)$ for $\beta$ similarly so that $\theta(s_0)=\tilde \theta(s_0)=0$.

Note that for any $l>s>s'>s_0$, $y(s)\neq y(s')$, otherwise the curve $(t,y(s))$ is part of $\a$ which is a geodesic. This is impossible, because $k>0$. Hence $y$ is increasing in $(s_0,l)$. So
\begin{equation}\label{e-1}
  x'=x\cos\theta, y'=x\sin\theta.
\end{equation}
Hence $\sin\theta\ge 0$. But for $s_0<s<l$, if $\sin \theta(s)=0$, then the geodesic $(t,y(s))$ is tangent to $\a$, which is impossible because of convexity of the region bounded by $\a$ and $\gamma$. So $\sin\theta>0$, there.

On the other hand, we have \cite[p. 253]{DoCarmo}:
$$
k=-\sin\theta+\theta'.
$$
Hence $0<\theta\le\pi$ on $(s_0,l)$. Similarly, we have
$$
\tilde k=-\sin\tilde \theta+\tilde\theta'.
$$
Since $k>\tilde k$ and $\theta(s_0)=\tilde\theta(s_0)=0$, so for $s>s_0$ near $s_0$, $\theta(s)>\tilde \theta(s)$. Suppose there is a first $l>s_1>s_0$ such that $\theta(s_1)=\tilde\theta(s_1)$. Then at $s_1$,
$$
k-\tilde k=\theta'(s_1)-\tilde\theta'(s_1)\le 0.
$$
This is impossible. Hence $0\le \tilde \theta(s)\le \theta(s)\le \pi$ in $(s_0,l)$.

Now
$$
\log x(l)-\log x(s_0)=\int_{s_0}^l \frac{x'}{x}ds=\int_{s_0}^l\cos\theta(s)ds
$$
and
$$
\log \tilde x(l)-\log \tilde x(s_0)=\int_{s_0}^l \frac{\tilde x'}{\tilde x}ds=\int_{s_0}^l\cos\tilde\theta(s)ds
$$
Hence $\log\tilde x(l)\ge \log x(l)=\log b$. Similarly, one can prove that $\log\tilde x(0)\le \log x(0)=\log c.$ In particular, $\tilde x(0)<\tilde x(l)$. Now the length $L(\gamma)$ of $\gamma$ is $\log b-\log c$. Hence
$L(\gamma)\le \log \tilde x(l)-\log\tilde x(0)$.

We claim that $\log \tilde x(l)-\log\tilde x(0)\le L(\sigma)$. We may assume $\tilde y(0)<\tilde y(l)$. Then $\log \tilde x(l)-\log\tilde x(0)$ is the length of the geodesic $(t,\tilde y(l))$, $\tilde x(0)<t<\tilde x(l)$. Then by the sine law in $\H^2$, we conclude that the claim is true. This completes the proof of the lemma.
\end{proof}
\begin{lem}\label{lem: curve} Let $\a$ be a closed geodesically convex curve in $\H^2$ with geodesic curvature $ k_\a> r>0 $. Let $\b$ be a geodesic circle with geodesic curvature $r$. Suppose $\a$ and $\b$ are tangent at $p$ such that $\a$ and $\b$ lie on the same side of the geodesic through $p$ and tangent to $\a$ and $\b$. Then $\a$ will lie inside $\b$.
\end{lem}
\begin{proof} We use the disk model for $\H^2$. We may assume that $\b$ is a Euclidean circle with center at the origin and with radius $a>0$, say. We may also assume that $p=(0,-a)$ and $\b$ is parametrized by $(a\cos\theta,a\sin\theta)$, $-\pi\le \theta\le\pi$.
 It is easy to see that $\b(\theta)$ is outside $\a$ near $p$,   for $\theta \in(-\frac\pi2-\theta_0,-\frac\pi2+
 \theta_0)=I$ for some $\theta_0>0$. Suppose the lemma is not true. Then $\b$ will intersect $\a$ at some $\theta_1\notin I$. Without loss of generality, we may assume that there is $\frac\pi 2\ge\theta_1\ge -\frac\pi2+
 \theta_0$, such that $\alpha$ and $\beta$ intersects at $q=\b(\theta_1)$ and $\b(\theta)$ lies strictly outside $\a$ in $(-\frac\pi2+
 \theta_0,\theta_1)$. Then the length of $\b$ from $p$ to $q$ is strictly larger than the length of $\a$ from $p$ to $q$ by the Gauss-Bonnet theorem and the fact that $k_\a>r$.  Then there is $\theta_1>\theta_2>-\frac\pi2+
 \theta_0$ such that the length of $\b$ from $p$ to $u=\b(\theta_2)$ is the same as the length   of $\a$ from $p$ to $q$. By Lemma \ref{Schur-l1}, we conclude that $d(p,q)\le d(p,u)$. Since $p,q,u$ are on the geodesic circle $\b$, this is impossible by the cosine law in $\H^2$.

\end{proof}
\begin{proof}[Proof of Proposition \ref{prop: outer}]
Let $p\in \Sigma$. Let $S$ be the geodesic sphere with radius $a$ which is tangent to $\Sigma$ at $p$ with the same unit outward normal  at $p$. Let $P$ be any normal section. That is, $P$ is the totally geodesic $\H^2$ which passes through $p$ and  contains the geodesic normal to $\Sigma$ (and $S$) at $p$. Let $\gamma=P\cap \Sigma$ and $\beta=P\cap S$.

Since the principal curvature of $\Sigma$ is larger than $\coth a$, $\gamma$ is a closed convex curve in $P$ with geodesic curvature larger than $\coth a$. $\beta$ is a geodesic circle of radius $a$ in $P$. By Lemma \ref{lem: curve}, $\gamma$ lies inside $\beta$ and hence is inside $S$. Since $P$ is an arbitrary normal section, the result follows.
\end{proof}

\section{Normalized embedding of $(S_r,\gamma_r)$}\label{sec: emb}

Let $(M^3,g)$ be an AH manifold satisfying Assumption A. Let $(S_r,\gamma_r)$ be as in   Lemma \ref{lem: R}. The isometric embedding of $(S_r,\gamma_r)$ is unique up to an isometry of $\H^3$. In order to prove the main results, we have to normalize the embedding. As a first step, using Lemmas \ref{lem: R} and  \ref{cor: H_0 upper bdd}, we can apply Propositions \ref{prop: inner} and \ref{prop: outer} to obtain the following:
\begin{lem}\label{lem: embed}
With the above assumptions and notations, we can find a positive constant $C$ such that for each small $r$, if $\Sigma_r$ is the isometric embedding of $(S_r,\gamma_r)$ in $\H^3$, then there exist geodesic balls $B_{in}$ and $B_{out} $ with the same center and radii $\rho_{in}$ and $\rho_{out}$ respectively, such that $B_{in}$ is in the interior of $\Sigma_r$, $B_{out}$ contains $\Sigma_r$ and $\rho_{in}$, $\rho_{out}$ satisfy:
\begin{equation}\label{eq: in and out}
\rho_{in}\ge \sigma -Cr^3, \rho_{out}\le \sigma+Cr^3,
\end{equation}
where $\sigma=\sigma(r)>0$ is given by $\sinh \sigma=\frac{1}{\sinh r}$.
\end{lem}
\begin{proof}
Let $r$ be a fixed small number. Let $\lambda_j(x)$ be the principal curvatures of $x\in\Sigma_r$. By Lemmas \ref{lem: R} and \ref{lem: H_0} and the Gauss equation,
 it is easy to see that
\begin{equation}\label{eq: lambda}
\lambda_j=\cosh r+O(r^5).
\end{equation}
Let $\coth \rho=\lambda_j$, then
\begin{equation*}
 \begin{split}
\rho=&\frac{1}{2}\log(\frac{\lambda_j+1}{\lambda_j-1})\\
=& \frac{1}{2}\log(\frac{\cosh r+1+O(r^5)}{\cosh r-1+O(r^5)})\\
=& \frac12  \log(\frac{\cosh r+1 }{\cosh r-1 })+O(r^3)\\
=& \sigma +O(r^3).
\end{split}
\end{equation*}
 From this and Propositions \ref{prop: inner} and \ref{prop: outer}, it is easy to see the corollary is true.
\end{proof}

By Lemma \ref{lem: embed}, the first normalization of the embedding is to normalize such that the center of the geodesic balls in Lemma \ref{lem: embed} is at a fixed point $o\in \H^3$. We will use geodesic polar coordinates   $(\sigma,y)$ with center at $o$, where $\sigma$ is the geodesic distance from $o$ and $y\in \S^2$ so that a point in $\H^2$ is of the form $\exp_{o}(\sigma y)$. The metric $g_{\H^2}$ is given by $d\sigma^2+\sinh^2\sigma~ g_0$ where $g_0$ is the standard metric on $\S^2$.

The isometric embedding $X^{(r)}$ is given by $X^{(r)}(x)=\exp_o(\sigma^{(r)}(x)y^{(r)}(x))$.

\begin{lem}\label{lem: spherical distance} With the above notations, there exists a constant $C>0$ such that for all  $r$ small enough,
\begin{equation*}
  \lf |d_{\S^2}(x_1,x_2)-d_{\S^2}\lf(y^{(r)}(x_1), y^{(r)}(x_2)\ri)\ri|\le C r^3
\end{equation*}
for $x_1, x_2\in \S^2$, where $d_{\S^2}$ is the distance on $\S^2$ with respect to the standard metric.
\end{lem}
\begin{proof} Let $x_1, x_2\in \S^2$ and let $X^{(r)}$ as above so that the embedded image $\Sigma_r$ lies between two concentric geodesic spheres $\pp B_o(R_1)$ and $\pp B_o(R_2)$ with center at $o$ and with   radii $R_1>R_2$ such that $R_i=\sigma+O(r^3)$, $i=1, 2$, and $\sigma$ is given by $\sinh\sigma=\frac{1}{\sinh r}$, by Lemma \ref{lem: embed}. Here and below $O(r^k)$ will denote a quantity with absolute value bounded by $Cr^k$ for some positive constant $C$ independent of $r$ and $x_1, x_2\in \S^2$.

 Let $l(x_1, x_2)$ be the intrinsic distance between $x_1, x_2\in S_r$ with respect to the metric $\gamma_r$. By the definition of AH manifold, it is easy to see that
 \begin{equation}\label{eq: l(x,y)domain}
    l(x_1, x_2)=\frac{1}{\sinh r}d_{\S^2}(x_1, x_2)\lf(1+O(r^3)\ri).
 \end{equation}
 On the other hand,   let $v_1, v_2$ be the points of intersections  of $\pp B_o(R_2)$ with   the geodesics  from $o$ to $X^{(r)}(x_1)$ and $X^{(r)}(x_2)$ respectively. Since $X^{(r)}$ is an isometric embedding, the intrinsic distance between $X^{(r)}(x_1)$ and $X^{(r)}(x_2)$ in $\Sigma_r$ is equal to $l(x_1, x_2)$. Since $\Sigma_r$ is strictly convex in $\H^3$ by \eqref{eq: lambda} and $R_i=\sigma+O(r^3)$,  we have
 $$
 l(x_1, x_2)\le d_{\pp B_o(R_2)}(v_1, v_2)+O(r^3)
 $$
 because $l(x_1, x_2)$ is the minimum of lengths of curves in $\H^3$ outside $\Sigma_r$ which join $X^{(r)}(x_1)$ and $X^{(r)}(x_2)$. Here $d_{\pp B_o(R_2)}$ is the intrinsic distance function on $\pp B_o(R_2)$. So we have
 \begin{equation*}
    l(x_1, x_2)\le \sinh \sigma d_{\S^2}\lf(y^{(r)}(x_1),y^{(r)}(x_2)\ri)+O(r^2).
 \end{equation*}
Using the fact that $\pp B_o(R_1)$ is also strictly convex, one can prove similarly,
\begin{equation*}
    l(x_1, x_2)\ge   \sinh \sigma d_{\S^2}\lf(y^{(r)}(x_1),y^{(r)}(x_2)\ri)+O(r^2).
 \end{equation*}
Combining these two inequalities we have:
\begin{equation}\label{eq: l(x,y)target}
    l(x_1, x_2)= \sinh \sigma d_{\S^2}\lf(y^{(r)}(x_1),y^{(r)}(x_2)\ri)+O(r^2).
 \end{equation}
By \eqref{eq: l(x,y)domain}, \eqref{eq: l(x,y)target} and the fact that $\sinh\sigma=\frac1{\sinh r}$, the result follows.
\end{proof}

Let $X^{(r)}$ be the isometric embeddings normalized as above.

\begin{lem}\label{lem: sphere} With the above notations, by composing $X^{(r)}$ with isometries of $\H^3$ fixing $o$, and the resulting isometric embeddings still denoted by $X^{(r)}$, we have:
$$\lim_{r\ra0} y^{(r)}(x)=x,\quad x\in \mathbb{S}^2.$$
The convergence is uniform in $x$.
\end{lem}

\begin{proof}   $x\in\S^2$ is of the form $x=(x^1,x^2,x^3)$ with $\displaystyle\sum_i(x^i)^2=1$. Let $e_1=(1,0,0)$, $e_2=(0,1,0)$ and $e_3=(0,0,1)$.
By composing with isometries of $\H^3$ fixing  $o$, we may arrange that for all $r$
\begin{equation}\label{eq: normalize}
y^{(r)}(e_1)=e_1, y^{(r)}(e_2)  \in \{x^3=0, x^2\geq 0\},  y^{(r)}(e_3)\in \{x^3\geq 0\}.
\end{equation}

By Lemma \ref{lem: spherical distance},
$$
d_{\S^2}(y^{r}(e_2),e_1)=d_{\S^2}(y^{r}(e_2),y^{(r)}(e_1))= d_{\S^2}(e_2,e_1)+O(r^3).
$$
By \eqref{eq: normalize}, we can conclude that $\displaystyle\lim_{r\to 0}y^{(r)}(e_2)=e_2$. For any $r_n\to0$ such that $y^{(r_n)}(e_3)\to a=(a^1,a^2,a^3)$ with $a^3\ge0$. Then by Lemma \ref{lem: spherical distance} again, we have
$$
d_{\S^2}(e_1,a)=d_{\S^2}(e_2,a)=\frac\pi2.
$$
Hence $a=e_3$. This implies that $\displaystyle\lim_{r\to 0}y^{(r)}(e_3)=e_3$. That is, we have
\begin{equation}\label{eq: e_i limit}
   \lim_{r\to\infty}y^{(r)}(e_i)=e_i,\ \text{ for $1\le i\le 3$}.
\end{equation}

Now for any $x\in \S^2$ and $r_n\to0$ such that $\displaystyle\lim_{n\to \infty}y^{(r_n)}(x)=b$. Then by \eqref{eq: e_i limit} and Lemma \ref{lem: spherical distance}, we have
$$
d_{\S^2}(e_i,b)=d_{\S^2}(e_i,x),\ \text{ for $1\le i\le 3$}.
$$
Hence $b=x$ and so $\displaystyle\lim_{r\to 0}y^{(r)}(x)=x$ for all $x\in \S^2$.

We claim that the convergence is uniform. Fix $x_0\in \S^2$ for any $\e>0$, by Lemma \ref{lem: spherical distance}, let $C$ be the constant in the lemma,  for any $x\in \S^2$ with $d_{\S^2}(x,x_0)<\e$, we have
\begin{equation*}
\begin{split}
    d_{\S^2}(y^{(r)}(x),x) \le & d_{\S^2}(y^{(r)}(x),y^{(r)}(x_0))+d_{\S^2}(y^{(r)}(x_0),x_0)+d_{\S^2}( x_0,x)\\
    \le & 2d_{\S^2}( x_0,x)+d_{\S^2}(y^{(r)}(x_0),x_0)+Cr^3\\
    \le & 3\e
    \end{split}
\end{equation*}
provided $r$ is small enough depending only on $x_0$ and $\e$. Since $\S^2$ is compact, this proves the claim that the convergence is  uniform.
\end{proof}

\section{Proofs of the main results}\label{sec: main results}

We now prove our main results. First, we embed $\H^3$ in the $\R^{3,1}$ so that
$\displaystyle\mathbb{H}^3=\{(x^0, x^1, x^2, x^3)\in \R^{3,1}: (x^0)^2-\sum_{i=1}^3 (x^i)^2=1, x^0>0\}$ and the fixed point $o$ in Section \ref{sec: emb} is mapped to the point $(1,0,0,0)$. 

\begin{proof}[Proof of Theorem \ref{thm: main}]
For $r$ small, let $X^{(r)}$ be the embedding  of $(S_r,\gamma_r)$ in $\H^3$ given by  Lemma \ref{lem: sphere}. With the notations as in section \ref{sec: emb}, when considered as an embedding of $(S_r,\gamma_r)$ in $\R^{3,1}$, $X^{(r)}$ is of the form
\begin{equation}\label{eq: embed in R^{3,1}}
 X^{(r)}(x)=(\cosh \sigma^{(r)}(x),\sinh\sigma^{(r)}(x)~y^{(r)}(x)).
\end{equation}
Now by Corollary \ref{cor: H_0-H}, Lemmas \ref{lem: embed} and \ref{lem: sphere}, we have as $r\to 0$,
\begin{equation}\label{eq: main estimates}
\begin{cases}
   H_0-H&= \frac {r^3}2\tr_{g_0}h+O(r^4),\\
     \cosh \sigma^{(r)}(x)&= \coth r +O(r^2)=\frac1r+o(1),\\
        \sinh \sigma^{(r)}(x)&= \frac1{\sinh r}+O(r^2)=\frac1r+o(1), \\
         y^{(r)}(x)&= x+o(1).
         \end{cases}
\end{equation}
As before, $O(r^k)$ represents a quantity with absolute value bounded by $Cr^k$ with $C$ being independent of $r$ and $x$. Moreover, by Definition \ref{def: g}, the volume form
\begin{equation}\label{eq: volume}
d\mu_{\gamma_r}=\lf(\frac1{\sinh^2r}+O(r^3)\ri)d\mu_{g_0}=(\frac{1}{r^2}+o(1))d\mu_{g_0}
\end{equation}
as $r\to 0$, where $d\mu_{g_0}$ is the volume form of the standard metric $g_0$.
By \eqref{eq: main estimates} and \eqref{eq: volume}, we have
\begin{equation*}
    \begin{split}
    &\int_{S_r}(H_0-H)X^{(r)}d\mu_{\gamma_r}\\
    =&\int_{S_r}(H_0-H)\lf(\cosh \sigma^{(r)}, \sinh \sigma^{(r)}~y^{(r)}\ri)  d\mu_{\gamma_r}\\
    =&\int_{\mathbb{S}^2}\left((\frac {r^3}2\tr_{g_0}h+O(r^4))
    \lf( \frac{1}{r}+o(1), \frac{x}{r}+o(\frac{1}{r})   \ri)(\frac1{r^2}+o(1))\right)d\mu_{g_0}\\
    =&\frac12 \lf( \int_{\mathbb{S}^2} \tr_{g_0}(h) d\mu_{g_0}, \int_{\mathbb{S}^2} \tr_{g_0}(h) xd\mu_{g_0}\ri)+o(1).
    \end{split}
\end{equation*}
>From this the theorem follows.
\end{proof}

\begin{proof}[Proof of Corollary \ref{cor: main}] Under the assumptions of the corollary, suppose $(M,g)$ is not isometric to $\H^3$, then by \cite[Theorem 2.5]{Wang}, or Theorem \ref{thm: wang},
$$
\int_{\S^2}\tr_{g_0}(h)d\mu_{g_0}> \lf|\int_{\S^2}\tr_{g_0}(h)xd\mu_{g_0}\ri|.
$$
Let $X^{(r)}$ be the isometric embedding of $(S_r,\gamma_r)$ as in Theorem \ref{thm: main}, then by the theorem there exists $\e>0$ such that if $r$ is small enough then for any future null vector $\eta=(1,\xi)$,
$$
\lf|\int_{S_r}(H_0-H)\langle X^{(r)},\eta\rangle_{\mathbb{R}^{3,1}} d\mu_{\gamma_r}\ri|_{\mathbb{R}^{3,1}}\le -\e .
$$
Hence $\int_{S_r}(H_0-H)  X^{(r)}  d\mu_{\gamma_r}$ is timelike and is future directed. From this and Remark \ref{rem}, it is easy to see that the corollary is true.
\end{proof}

\bibliographystyle{amsplain}

\end{document}